\documentclass[12pt]{amsart}
\usepackage{latexsym,fancyhdr,amssymb,color,amsmath,amsthm,graphicx,listings,comment}
\usepackage[section]{placeins}
\newtheorem{thm}{Theorem} 
 
\newtheorem{lemma}{Lemma} 
\newtheorem{coro}{Corollary}
\newtheorem{conj}{Conjecture}
\let\paragraph\subsection

\title{Remarks about Connection and Dirac matrices}
\author{Oliver Knill}
\date{January 25, 2026, Updated July 29, 2026}

\address{Department of Mathematics \\ Harvard University \\ Cambridge, MA, 02138 }

\begin{document}
\begin{abstract}
The connection Laplacian $L$ and the Dirac matrix $D$ are both 
$n \times n$ matrices defined from a given finite abstract simplicial complex $G$ with $n$ sets. 
In both cases, there is interlacing of the eigenvalues for subcomplexes.
This gives general upper bounds of the eigenvalues $\lambda_j \leq d_j$ both for
$L$ and $D$ in terms of inclusion or intersection degrees. 
We conjecture that $L$ always dominates both $D$ and $g=L^{-1}$ in a weak Loewner sense.
In a second part we look at dynamical systems $(G,T)$, where $T$ is a simplicial map on $G$.
Both $L$ and $D$ generalize to dynamical versions $L_T$ and $D_T$.
$L_T$ is still unimodular with an explicit Green function inverse $g_T$ and 
$D_T=d_T + d_T^*$ still comes from an exterior derivative $d_T$. 
We also review the Lefschetz fixed point theorem $\sum_{x \in {\rm F}} i_T(x) = \chi_T(G)$ 
for a simplicial map $T$ on a finite abstract simplicial complex $G$. It implies
the Brouwer fixed point theorem telling that simplicial map on a contractible finite abstract
simplicial complex $G$ always has at least one simplex fixed. This is all combinatorial as 
no geometric realizations are ever involved. 
\end{abstract}

\maketitle

\section{Introduction}

\paragraph{}
If $G$ is a {\bf finite abstract simplicial complex} with $n$ sets, both the 
{\bf Dirac matrix} $D$ as well as the {\bf connection matrix} $L$ are 
$n \times n$ matrices.  The Dirac matrix $D=d+d^*$ uses the {\bf exterior derivative} $d$
is a matrix $d(x,y) = {\rm sign}(x,y)$ if $|x|$ and $|y|$ differ by $1$ and $y \subset x$
and $d(x,y)=0$ else. The {\bf connection matrix} is defined as $L(x,y)=1$ if 
$x\cap y \neq \emptyset$ and $L(x,y)=0$ else. Its inverse $g=L^{-1}$ defines {\bf Green functions} 
$g(x,y)$.  Both the {\bf Hodge Laplacian} $H=D^2=d d^* + d^* d$ and the connection Laplacian 
$L^2$ have non-negative spectrum. For $1$-dimensional complexes, 
$L-L^{-1}=|D|^2$. \cite{HearingEulerCharacteristic}. 

\begin{figure}[!htpb]
\scalebox{0.85}{\includegraphics{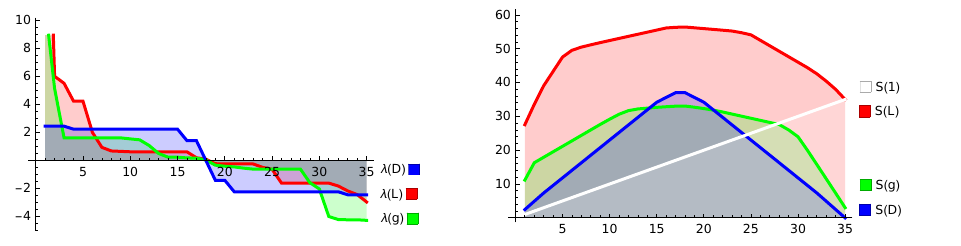}}
\label{Connection, Dirac an Green}
\caption{
The figure show the spectra $\lambda$ (left) and cumulative spectra $S$ (right) of the connection
matrix $L$, the Green matrix $g=L^{-1}$ and the Dirac matrix $D=d+d^*$ for a random complex $G$.
$L$ appears to weakly Loewner dominate both $D$ and $g$.
}
\end{figure}

\paragraph{}
The Dirac matrix uses ``inclusion" $x \subset y$ while the connection matrix uses ``intersection"
$x \cap y \neq \emptyset$ for its definition. Both frame-works have spectral relations to the 
{\bf Euler characteristic} $\chi(G) = \sum_{x \in G} \omega(x)$ with $\omega(x)=(-1)^{{\rm dim}(x)}$. 
The kernels ${\rm ker}(H_k)$ of the Hodge Laplacian blocks $H_k$ are the {\bf cohomology groups} of $G$
and their dimensions ${\rm dim}({\rm ker}(H_k))$ are the {\bf Betti numbers} $b_k$.
By Euler-Poincar\'e, they super-sum to the Euler characteristic $\sum_{k=0}^q (-1)^k b_k=\chi(G) = 
\sum_{x \in G} \omega(x)$. This identity can already be seen as an example of the Lefschetz fixed point
formula, where the simplicial map is the identity and every $x \in G$ is a fixed point and the 
Euler characteristic is the Lefschetz number. Remarkably, the connection matrix $L$ and so its square $L^2$ 
are both unimodular. The determinant is explicitly known as 
${\rm det}(L)=\prod_{x \in G} \omega(x)$. 
The Euler characteristic is the total potential energy $\sum_{x,y \in G} g(x,y) = \chi(G)$.
The number of positive eigenvalues 
minus the number of negative eigenvalues of $L$ is equal to the Euler characteristic. For $D$ the number of positive
and negative eigenvalues is the same. If the spectra of $D$ and $L$ and the inverse $g=L^{-1}$ are plotted together then 
the graphs $j \to \lambda_j(D)$ and $j \to \lambda_j(L)$ and $j \to \lambda_j(g)$ cross close to each other. 
The cumulative spectra appear to be ordered: the cumulative sum $S_j(L)$ appears to dominate both $S_j(g)$
and $S_j(D)$.

\paragraph{}
The Hodge Laplacian $H=D^2$ is always a {\bf reducible} matrix because it is block diagonal.
It is also the sum of $D_k=d^* d$ and $D_k' =d^* d$ which are called a
{\bf combinatorical Laplacians}. The $D_k$ plays a role for {\bf analytic torsion}
$A(G) = \prod_k {\rm Det}(H_k)^{k (-1)^{k+1}}$, where
${\rm Det}$ is the {\bf pseudo determinant}, the product of the non-zero eigenvalues.
Torsion $A(G)$ can be rewritten as a {\bf super pseudo determinant}
${\rm SDet}(D) = \prod_k {\rm Det}(D_k)^{(-1)^k}$ of these Dirac blocks. Torsion satisfies
$A(G)=|V|$ for contractible complexes, $A(G)=|V|/|V'|$ for even dimensional spheres and 
$|V| |V'|$ for odd dimensional spheres, where $|V|$ is the number of zero-dimensional simplices
and $|V'|$ is the number of facets, the highest dimensional simplices. Unlike in the continuum,
where the subject is quite technical, analytic torsion is a topic for linear algebra 
in the finite \cite{KnillTorsion}.

\paragraph{}
In contrast to the Hodge Laplacian $H=D^2$, the connection matrix $L$ and so its Laplacian 
$L^2$ are always irreducible non-negative $0-1$ matrices if $G$ is connected. The reason is that $L$ has
only non-negative entries and that the connection graph of $G$ has a finite diameter $r$ implying that $L^r$
has only positive matrix entries implying irreducibility. This implies that $L$ (and so $L^2$) 
always has a unique {\bf Perron-Frobenius eigenvalue}. As $L$ is invertible with determinant $1$ or $-1$, 
also its inverse $g=L^{-1}$ is an irreducible integer matrix. An interesting question which we can not 
answer yet is whether also $g=L^{-1}$ always has a unique maximal eigenvalue. The powers of $g$ do not necessarily
become positive now.  But we see a unique maximal eigenvalue in the case in all examples considered so far. 
While the trace of $L$ is $n$, the trace of $g$ is the sum of the Euler characteristics 
$\sum_x \chi(U(x))$ of the smallest open sets $U(x)=\{y \in G, x \subset y\}$. 
This follows from the general Green-Star formula $g(x,y) = \omega(x) \omega(y) \chi(U(x) \cap U(y))$. 

\paragraph{}
When looking at the {\bf fusion inequality} 
$b(U)+b(K) \geq b(G)$ \cite{fusion1,fusion2} for an open closed pair
$U \cup K =G$, where $U$ is open and $K$ is closed, we made use of the fact that the 
spectra move monotonically for the Hodge spectra. This had used the 
{\bf Courant minimax principle}. But more is true: the spectra also interlace.
As we will see, in both cases, we can get from $L_G$ to $L_K$ or from $D_G$ to $G_K$ 
by deleting some rows or columns belonging to maximal simplices, then
use the Cauchy interlace theorem \cite{HornJohnson2012}:
if $\mu_k$ are the eigenvalues of a principal $(n-1 \times n-1)$ sub-matrix of $K$,
then $\lambda_k \leq \mu_k \leq \lambda_{k+1}$.

\paragraph{}
We suspect that $L \geq D$ in a {\bf weak Loewner sense}, meaning that the 
eigenvalues $\lambda_j$ of $L$ and the eigenvalues $\mu_j$ of $D$ satisfy 
$\sum_{j=1}^k \lambda_j \geq \sum_{j=1}^k \mu_j$ for all $k$. We also see
that $L$ weakly Loewner dominates its inverse $g$. These are open questions. 

\paragraph{}
In the 1-dimensional case, there is the relation $|H|=L-L^{-1}$ between the 
{\bf sign-less Hodge Laplacian} $|H|=|D|^2$ and the connection Laplacian.
We called it the {\bf hydrogen identity} \cite{Hydrogen}.
For Kirchhoff spectra of graph, this can be useful for estimating
the spectral radius of $|H|$ and so give upper bounds for the spectral 
radius of $H$ in terms of the spectrum of $L$.

\paragraph{}
For {\bf higher characteristics} $w_k(G)$ \cite{CharacteristicTopologicalInvariants},
which generalize Euler characteristic $w_1(G)=\chi(G)$, there are parallel stories. 
In each case, there are {\bf connection Green functions}
$g(x,y) = w(x) w(y) w_k(U(x) \cap U(y)$. It can be generalized to energized cases. 
The connection story is a 2-point Green function story and deals with matrices so that it 
is linear algebra. There are {\bf k-point Green functions}.
For higher characteristic, the matrix $g(x,y)=w(x) w(y) w_k(U(x) \cap U(y))$ is not 
unimodular any more in general. 

\section{The Connection Laplacian}

\paragraph{}
If $G$ is a finite abstract simplicial complex containing $n$ sets, 
the $n \times n$ connection matrix $L$ can be defined as $L(x,y)=\chi(C(x) \cap C(y))$, 
where $C(x) = \{ y \in G, y \subset x\}$ and where
$\chi(A) = \sum_{y \in A} \omega(x)$ is the {\bf Euler characteristic}
with $\omega(x)=(-1)^{{\rm dim}(x)}$. This more general definition is useful for generalized
energized versions, where $\omega(x)$ is an arbitrary non-zero function (which even can become
division algebra-valued \cite{EnergizedSimplicialComplexes2}).
In the special case of the Euler characteristic, 
we have $L(x,y)=1$ if $x,y$ intersect and $0$ else, because $C(x) \cap C(y)$ 
is the simplicial complex that is defined by the simplex $x \cap y$ and $\chi(C(x))=1$. 

\paragraph{}
The matrix $L$ is unimodular and has as inverse the integer matrix
$g(x,y) = \omega(x) \omega(y) \chi(U(x) \cap U(y))$, where 
$U(x)=\{ y \in G, x \subset y\}$ is the {\bf star} of $x$. The 
determinant of $L$ is known to be the {\bf Fermi characteristic} $\prod_{x \in G} \omega(x)$, which is 
a multiplicative version of the Euler characteristic $\chi(G) = \sum_{x \in G} \omega(x)$. 
We know that the number of positive eigenvalues of $L$ minus the number of negative 
eigenvalues of $L$ is the Euler characteristic of $G$. See \cite{KnillEnergy2020}. 

\paragraph{}
We have seen already that if $K \subset G$ is a sub-complex of $G$, then $L(K)$ is a 
principal sub-matrix of $L(G)$ so that interlacing happens. 
This is related to the result (see \cite{Godsil} ) that the 
non-zero eigenvalues of a graph $G$ interlace the eigenvalues 
of the graph $G$ in which an additional edge has been added. 
As $L(K)$ still is a connection matrix of a complex, it still is unimodular.

\paragraph{}
As for open subsets $U \subset G$, and if we take a vertex away, we still have 
an open set but we do not have a principal submatrix any more in general. The 
removal of a vertex also removed connections between higher dimensional simplices. 
Take the example of $U=\{ \{1\}, \{1,2\} \}$ in $K_2$. Then, $L(U)$ is a $2 \times 2$
matrix, where each entry is $1$. This obviously is no more unimodular. 

\paragraph{}
Define the {\bf connection degree} of $x \in G$ as 
$d(x) = | \{ y \in G, \chi(C(x) \cap C(y))  \neq 0 \} |$, where 
$C(x)= \{ y \in G, y \subset x \}$. 
In the $1$-dimensional case, this is the vertex degree if $x$ is a vertex.
We also have $d((a,b))= d(a)+d(b)$ if $x=(a,b)$ is an edge of the 1-dimensional complex. 
Let $d_1 \geq d_2 \cdots $ denote the ordered connection degrees of $G$ and 
$\lambda_1 \geq \lambda_2 \cdots$ the ordered connection eigenvalues of $L$. 

\begin{thm}
The connection Laplacian spectrum satisfies $\lambda_j \leq d_j$. 
\end{thm} 

\begin{proof} 
Use induction with respect to the number of elements in $G$. 
Removing a maximal simplex $x$ has the effect that the connection matrix is
a principal sub-matrix. Let $\mu_j$ be an eigenvalues of 
the smaller reduced complex $K$. Induction tells $\mu_j(K) \leq d_j(K) \leq d_j(G)$. 
The spectral radius estimate $\lambda_1 \leq d_1$ follows from basic principles \cite{HornJohnson2012}:
the spectrum $L$ is majorized by the spectrum of $|L|$ which is bounded above by the 
maximal row or column sum which is $d_1$. 
Now $\lambda_2 \leq \mu_1 \leq d_1(K) \leq d_2(G)$ and
    $\lambda_3 \leq \mu_2 \leq d_2(K) \leq d_3(G)$ etc.  
\end{proof} 

\paragraph{}
This does not give any information for negative eigenvalues 
because $d_j \geq 0$.  An other inequality is the Schur-Horn inequality 
$\sum_{j=1}^k \lambda_j \leq k$ which follows because the diagonal entries of $L$ are $1$. 
This concludes the remark about the majorization of eigenvalues
of the connection matrix in simplicial complexes. We turn now to the Dirac case. 

\begin{figure}[!htpb]
\scalebox{0.85}{\includegraphics{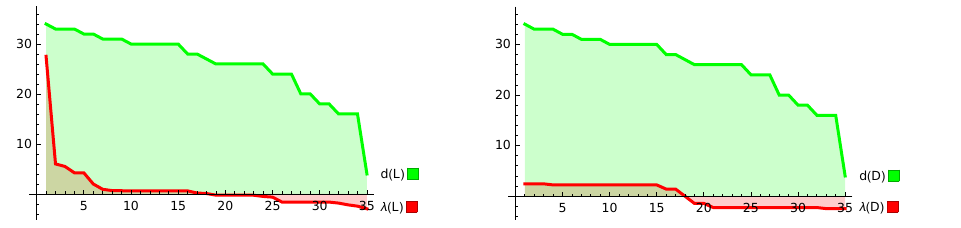}}
\label{Comparison}
\caption{
The estimates $\lambda_j(D) \leq d_j(D)$ and $\lambda_j(L) \leq d_j(L)$
where $d_j(D)$ and $d_j(L)$ are the ordered vertex degrees of $D$ and $L$.
}
\end{figure}

\section{The Dirac matrix}

\paragraph{}
The interlacing story has a parallel situation for the Dirac matrix $E$
Let $D_K$ denote the Dirac matrix of $K$ if $D=D_G$ is the Dirac matrix of $G$. 
In the Dirac case, the interlacing not only works for closed sets, it also does for 
open sets: 

\begin{thm}
If $K \subset G$ is a sub-complex of $G$, then the eigenvalues of $D_K$ interlace the
eigenvalues of $D_G$. The same holds if $U \subset G$ is an open set. 
\end{thm}

\begin{proof}
The proof is the same and reflects the fact that going to a subcomplex can be 
achieved by removing successively maximal simplices in $G$ that are not in $K$. 
As for open sets, removing a maximal simplex $x$ only changes the matrix entries $D(x,y)$ of smaller
dimensional simplices $y$ contained in $x$. Unlike for $L$ (where intersection matters) 
it does not affect other entries.
\end{proof} 

\paragraph{}
The estimate is very rough and far from optimal. 
Estimating the individual eigenvalues of $H=D^2$ is trickier. It was possible in the 
case of $H_0$, the Kirchhoff matrix \cite{Knill2024}. By McKean-Singer, the non-zero 
eigenvalues of $H$ restricted to even forms is the set set of non-zero
eigenvalues of $H$ restricted to odd forms. Eigenvalues therefore come in pairs. 
This can be encoded in short as
${\rm str}(H^n)=0$ for all $n>0$. We also have ${\rm str}(H^0)={\rm str}(1) = \chi(G)$. The Taylor expansion
of the exponential function immediately gives the McKean-Singer symmetry ${\rm str}(e^{-it H}) = \chi(G)$. 

\paragraph{}
Define the {\bf Dirac graph} of a simplicial complex has the graph in which 
the elements in the complex as vertices and where two vertices are connected 
if $D(x,y)$ is non-zero. Define the {\bf Dirac graph degrees} as the vertex degrees of this graph. 
We can rewrite this as $d_x=\sum_k |D_{xy}|$. Again we assume that the eigenvalues
are ordered in a descending manner and relabel the sequence as $d_1 \geq d_2 \geq \cdots \geq d_n$
if $|G|=n$. Also the eigenvalues $\lambda_j$ of $D$ are ordered as $\lambda_1 \geq \lambda_2 \cdots \geq \lambda_n$. 

\begin{thm}
The ordered eigenvalues $\lambda_j$ of the Dirac matrix of a complex satisfy $\lambda_j \leq d_j$, 
where $d_j$ are the ordered Dirac graph vertex degrees.  The same works for open sets in a complex. 
\end{thm}

\begin{proof}
The proof is identical to the connection case and works with induction to
the number $n$ of simplices of $G$. In the same way one can estimate the eigenvalues of $D$ first
from above with the eigenvalues of $|D|$ in which all matrix entries are non-negative, then 
use the row sums $d_x= \sum_{y \in G} |D_{xy}|$ and have that the maximum of these sums is an 
upper bound for the spectral radius $\lambda_1$.  
\end{proof} 

\section{Wave equations}

\paragraph{}
As $D,L$ and $g=L^{-1}$ can have both positive and negative eigenvalues, it makes sense to 
compare them and see them on the same footing. Their squares $D^2$ and $L^2$ or $g^2$ produce then
``Laplacians", self-adjoint matrices that have non-negative eigenvalues only. They both can 
serve as generators of a {\bf Dirac wave equation} $u_{tt} = -D^2 u$, as a
{\bf connection wave equation} $u_{tt} = - L^2 u$, or then the {\bf Green connection wave equation}
$u_{tt} = -g^2 u$. These equations are equivalent to 
Schr\"odinger evolution equations $u_t = \pm i D u$ in the Dirac case or
$u_t = \pm i L u$ in the connection case or $u_t = \pm i g u$ in the Green case. 

\paragraph{}
In any of these situations, we have explicit solutions of the wave equations. The only thing which is needed
for such a d'Alembert solution is the ability to see the Laplacian as a square of an other operator. 
Define first the {\bf sinc function} ${\rm sinc}(x) = \sin(x)/x$. It is an entire function as one can 
see by dividing the Taylor series of $\sin$ by $x$. 

\begin{thm}
a) The Dirac wave equation is solved by $u(t) = \cos(D t) u(0) + t \; {\rm sinc}(D t) u'(0)$.
b) The connection wave equation is solved by $u(t) = \cos(L t) u(0) + t \; {\rm sinc}(L t) u'(0)$.
c) The Green wave equation is solved by $u(t) = \cos(g t) u(0) + t \; {\rm sinc}(g t) u'(0)$.
\end{thm}

\begin{proof} 
Both $\cos$ and ${\rm sinc}$ are entire functions. We obviously have for $t=0$ the initial condition $u(0)$ right.
Now note that $t {\rm sinc}(D t)$ has at $t=0$ the derivative $1$ so that also $u'(0)$ matches. 
\end{proof} 

\paragraph{}
If $||L t|| <\pi/2$ (rsp. $||g t||<\pi/2$ or $||D t||<\pi/2$), the inverse problem of getting for any 
vector $q$ and initial position $p=u(0)$, we can find an initial velocity $u'(0)$ such that $u(t) = q$. 
Scaling $t$ to $t/c$ obviously is equivalent to scaling $D$ to $D c$. 

\begin{coro}
Assume $cD$ as norm smaller than $\pi/2$, then the wave equation  $u''= -c^2 D^2 u$ 
has for a given initial end time boundary conditions $u(0)$ and $u(1)$ a unique solution. 
The same statement holds if $D$ is replaced by $L$ or by $g$. 
\end{coro} 
\begin{proof} 
Use the solution formula $u(t) = \cos(c D t) u(0) + t {\rm sinc}(c D t) u'(0)$
which gives $u'(0) = {\rm sinc}(c D)^{-1} [u(1) - \cos(cD) u(0) ]$. Because
$cD$ had norm smaller than $\pi/2$ the operator ${\rm sinc}(cD)$ is invertible. 
\end{proof} 

\begin{figure}[!htpb]
\scalebox{0.55}{\includegraphics{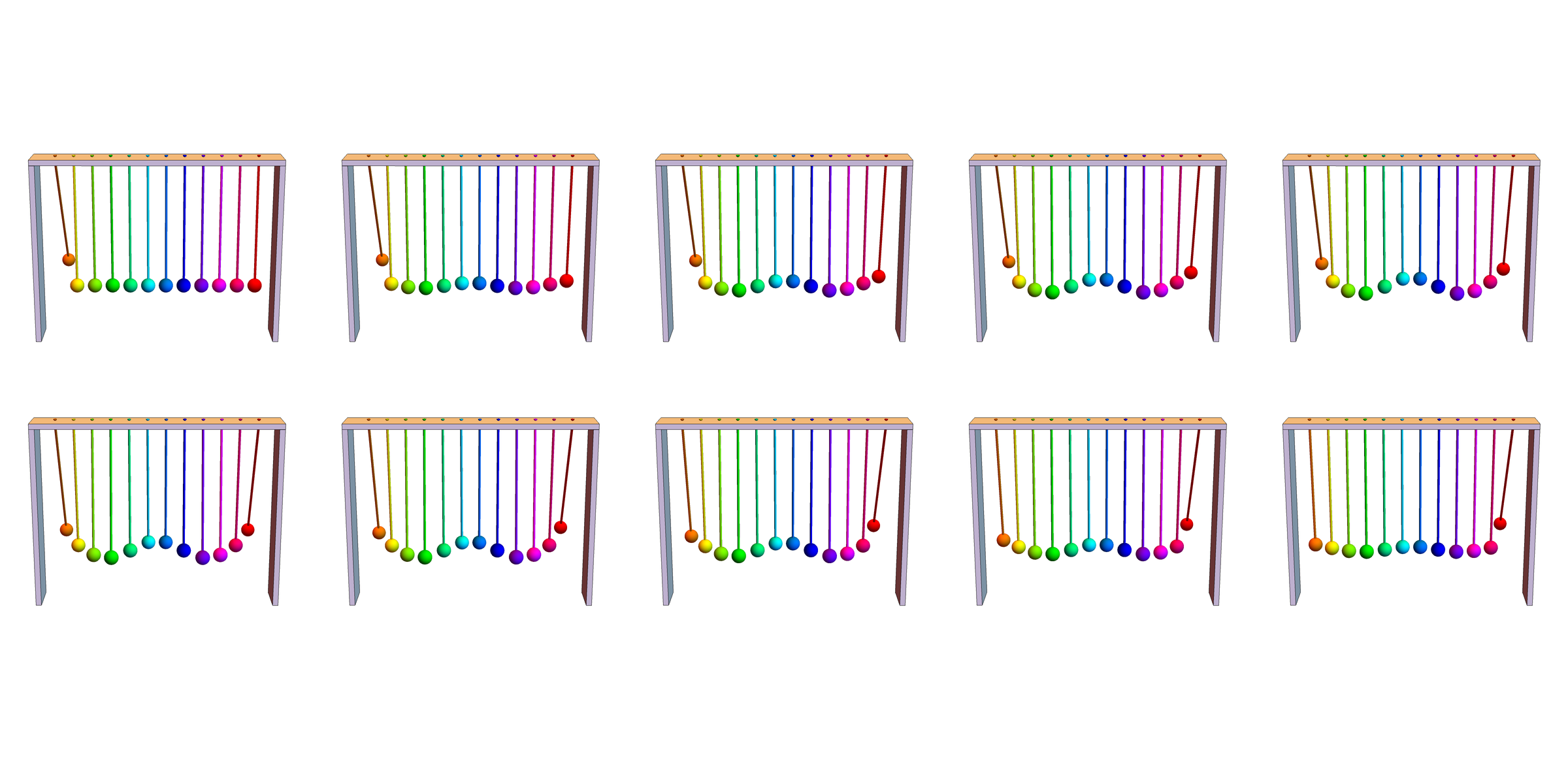}}
\label{Craddle}
\caption{
A Newton craddle toy illustrates the interpolation result.
We can have the first pendulum to be the only nonzero at t=0
and the last pendulum to be the only nonzero one at t=1.
We can hit all the penduli with a suitable initial velocity
so that the wave interpolates between the two positions.
}
\end{figure}

\begin{figure}[!htpb]
\scalebox{0.55}{\includegraphics{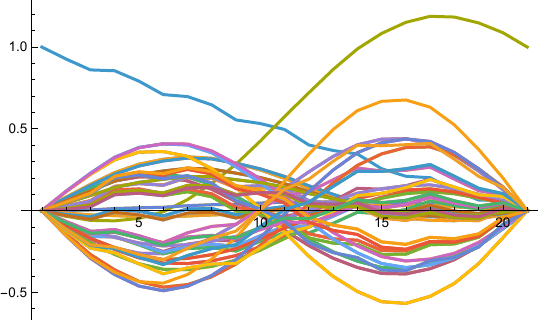}}
\scalebox{0.55}{\includegraphics{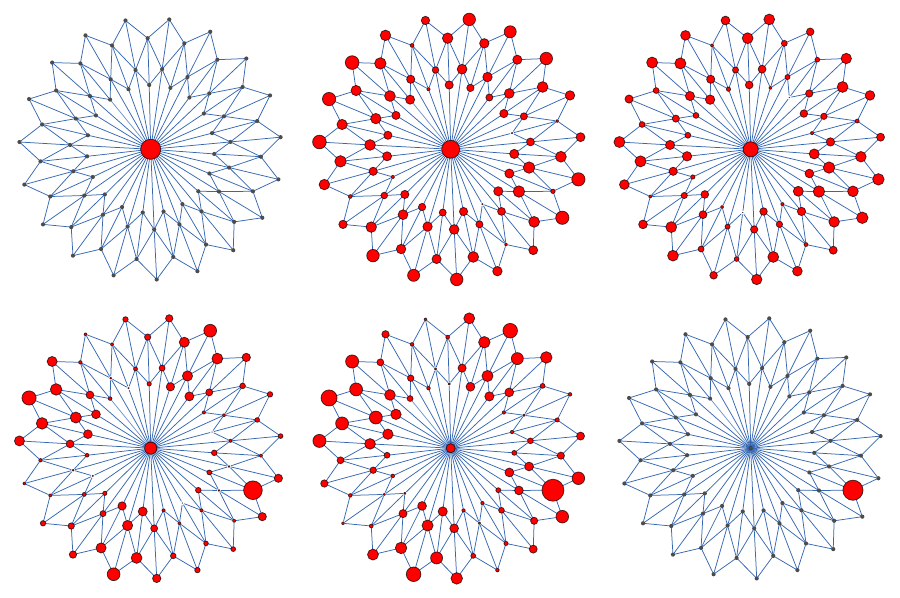}}
\label{Wave}
\caption{
We see solution of the wave equation $u_{tt} = -Lu$ on a finite geometry that starts
localized at one point and ends at the other point. The initial velocity $u'(0)$ is explicitly given.
We have visualized the wave $\{u(x,t)\}_{x \in G}$ by scaling the vertex
size of each vertex $x$ in the graph $(G,\{ (x,y), x \subset y || y \subset x\}$
belonging to the finite abstract simplicial complex $G$.
}
\end{figure}

\paragraph{}
If $u(0)$ is localized on one simplex $x \in G$ and $u(1)$ is localized on
a second simplex $y \in G$ and $x,y$ are far remote. How come one can travel in time $t=1$ from
$x$ to $y$? The paradox is resolved by noting that the initial velocity $u'(0)$ needed to go from $x$
to $y$ is not localized at $x$. It spreads over the entire stretch from $x$ to $y$. We even can get
from $x$ to $y$ if $x$ and $y$ are in different connectivity components of $G$: what will happen is that
in one of the components we have an initial velocity $u'(0)$ and position $u(0)=e_x$ so 
that $u_1(x)=0$ for all $x$ in the first component.  In the other connectivity component, 
we have that $u_0=0$ and $u_1=e_y$. There is an initial velocity which gets us there. 

\paragraph{}
We should also note that for wave equations with continuum time $t \in \mathbb{R}$ and discrete
space $G$, there is no finite propagation speed. There is no Huygens principle like in the 
continuum where one has a strong Huygens principle in odd dimensions or weak Huygens principle in even dimensions
(see for example \cite{Evans2010} page 73). If we write the solution to the wave equation $u_{tt} = - D^2 u$
in terms of eigenfunctions, then these eigenfunctions are global. Removing a simplex $x$ at the end
of the universe $G$ impacts the eigenvalues and eigenfunctions of the Laplacian $D^2$ and so impacts 
the solution to the wave equation. The reason is that if $G$ all solutions are entire functions in $t$; 
we even wrote down the solutions explicitly in terms of $\cos$ and ${\rm sinc}$ functions.

\paragraph{}
We have just seen that starting a wave so that $u(0)$ is localized at some simplex $x$ and 
zero everywhere else. Then the solution of $u(t)$ involves every 
other simplex in the complex, even at the end of the ``universe" $G$. However, by going to discrete time, we 
achieved to have a causal wave evolution. At time $t$, only simplices up to distance $t$ matter. 
This {\bf causality issue} can be important when making numerical experiments because we like to 
experiment with parts of the world $G$ which are manageable and know that parts far away do not 
affect a given simplex for some time. In order to understand the spectral 
properties in an infinite lattice for example, we can explore the problem by looking at solutions of 
the Schr\"odinger equations as the spectral measures are directly related to dynamical properties
See \cite{Kni97,Kni98}. To solve the {\bf causality problem}, we need also to {\bf discretize time}.

\paragraph{}
If $L$ or $D$ or $g$ are of norm smaller than $1$, we can now look at the
{\bf symplectic wave map} $(u,v) \to (2 L u-v,u)$ or $(u,v) \to (2Du-v,u)$ or
$(u,v) \to (2g u-v,u)$ which are conjugated to a pair of unitary evolutions $(u,v) \to (U u, U^* u)$.
These are ``cellular automata" maps for Laplacians $H$ like $D^2,L^2$ or $g^2$ because in all these
cases matrix entries $H(x,y)$ are zero if $U(x),U(y)$ do not intersect. 
The discretization of time comes with a change of the operator $L \to \tilde{L}$ but the spectral 
type of the operator does not change. Of course, if $G$ is finite, there are just eigenvalues and there
is no singular or absolutely continuous spectrum. The change to a discrete time however solves the
causality issue and the wave needs time to get from $x$ to $y$. If $x,y$ are far away, then an initial
wave for which the initial times $u_0,u_{1}$ are localized near $x$ needs some time to reach $y$ and
this depends on the distance in the Dirac graph, or Connection graph or Green graph. It looks at first
like a philosophical issue, but it is helpful in numerical simulations. It is also nice to see actual
wave equations interpolated by cellular automata maps. 

\begin{thm}
If $||cD||<1$, there is a symplectic matrix $(u,v) = (2Du-v,u)$ on $\mathbb{R}^{2n}$ that
interpolates $(u_t,v_t) = i (c \tilde{D}, -c \tilde{D})$.  The same statement holds if 
$D$ is replaced by $L$ or by $g$ and the matrices are scaled such that $||cL||<1$ or that $||c g||<1$.
\end{thm}

\begin{proof} 
In \cite{Kni98} we scale $L$ to have norm smaller than 1, then took
$\tilde{L}=\arccos{L}$ and use $L=\cos(\tilde{L})$ to get
$U^n=\exp(i n \tilde{L})) = \cos(n \arccos(L)) + i \sin(n \arccos(L)) = T_n(L) +  i R_n(L)$
with {\bf Chebychev polynomials} $T_n(x) = \cos(n \arccos(x)), R_n(x) = \sin(n \arccos(x))$.
The unitary $U = L + i \sqrt{1-L^2}$ solves $U+U^*=2L$. The two evolutions
$$ B= \left[ \begin{array}{cc}  2D & -1 \\ 1 & 0 \\ \end{array}  \right], \;
   A= \left[ \begin{array}{cc}  U & 0 \\
                                0 & U^* \\ \end{array}  \right]   \;   $$
are conjugated via $B=S^{-1} A S$ using $S= \left[ \begin{array}{cc} U^* & U \\ 1   & 1 \\ \end{array} \right]$.
While $u(n) = T_n(L) u(0) \pm R_n(L) u'(0)$ are not unitary, they are conjugated to unitary evolutions. \\
We also mentioned that $V=-i U$ solves $i (V-V*) = 2L$ if $U+U^* = 2 L$ so that one can replace
$\cos$ with $\sin$ which makes the similarity more clear. We have then compared three cases
$$  \exp(i \arcsin(2 h L)) = 1-2 i h L - 2 h^2 L^2 - 2h^4 L^4 ...   $$
$$  \exp(2 i h L) = 1-2i h L - 2 h^2 L^2+(4/3) i h^3 L^3 ....     $$
$$ (h L - i)/(h L + i) = 1-2 i h L - 2 h^2 L^2 + 2 i h^3 L^3 ....  $$
where the later is the Cayley method. They all agree to second order $1-2i h L - 2 h^2 L^2$. 
\end{proof}

\paragraph{}
It is worth noting 
that the symplectic discretization of the wave equation is a time 1 step of an actual wave equation, just for a slightly
deformed operator $\tilde{L}= \arcsin(L)$ or $\tilde{D}= \arcsin(D)$ or $\tilde{g} = \arcsin(g)$
and it is causal as every time step is of a cellular automaton type and so satisfies a weak 
Huygens principle (the strong Huygens principle is a miracle that happens only for odd dimensional manifolds).
We therefore also have explicit solutions to these discrete time wave equations like
$u(t) = \cos(t \tilde{D}) u(0) + t {\rm sinc}(t \tilde{D}) u'(0)$. How come that the above 
causality does not work here? It would, if we look at the solution for 
general $t \in \mathbb{R}$. But we only look for integer $t$: indeed, the
unitary evolution is for integer $t$ given by Chebyshev polynomials for integer $t$. 
$e^(\pm i t arcsin(2h L) ) = T_t(L) \pm i R_t(L)$. For $t=1.5$ however the solution $u(t)$ 
would depend on all the space $G$ and even on simplices at the other and of $G$. 

\section{Weak Loewner order}

\paragraph{}
Assume $A$ is a self-adjoint $n \times n$ matrix with eigenvalues 
$\lambda_1 \geq \lambda_2  \geq \cdots \geq \lambda_n$ ordered in a decreasing manner.
Define the {\bf spectral sum} $S_k(A) = \lambda_1 + \lambda_2 + \cdots + \lambda_k$. The relation 
$S_k(A) \leq S_k(B)$ for all $1 \leq k \leq n$ defines a {\bf partial order} on 
the set of self-adjoint matrices. (As for reflexivity, note that $S_1(A) \leq S_1(B)$
and $S_1(B) \leq S_1(A)$ implies $\lambda_1(A)=\lambda_1(B)$. 
Then use $S_2(A) \leq S_2(B)$ and $S_2(B) \leq S_1(A)$ that 
$\lambda_2(A) = \lambda_2(B)$ etc. As this order is weaker than 
the {\bf Loewner order}, we call it the {\bf weak Lowener order}.  

\paragraph{}
The following was our starting point for this note as we aimed to prove that $L$ spectral dominates
$D$ in a weak Loewner sense. The research then then meandered in the fall of 2025 to other topics.
The entry point was was a {\bf  Brouwer conjecture} project from the summer 2025 \cite{knillbrouwer}.
It is a line of research, where the spectral sum $S_k(A)$ of a matrix is the focus of attention. 
The Brouwer conjecture \cite{Brouwer} deals with the spectra of the {\bf Kirchhoff matrix} $H_0(G)$ of a complex. 
The conjecture is $S_k \leq B_k$ for the spectral sum $S_k = \sum_{j=1}^k \lambda_j$ of the Kirchhoff matrix
is bounded above by the {\bf Brouwer bound} $B_k(G)=m+k(k+1)/2$, where $m=f_1(G)$ is the number of edges in 
$G$. 

\paragraph{}
When looking at the spectra, it is natural to compare the spectra of $L$ and $D$ (as seen in Figure~1). 
The difference $L-D$ is not positive definite in general. We can ask  a different question however:
is $L-D$ is positive in a {\bf weak Loewner order}? Experiments suggest that this is the case, but we can not prove
it yet. Since the diagonal elements of $L$ are $1$, we know that from the Schur inequality that $S_k(L) \geq k$. 
Schur also gives $S_k(D) \geq 0$ which is obvious as the eigenvalues of $D$ are symmetric: if $\lambda$
is an eigenvalue, then $-\lambda$ is an eigenvalue. 
We see evidence however that $L$ is larger than $D$ in the weak Loewner sense: 

\begin{conj}
If $L$ is the connection matrix and $D$ the Dirac matrix, then $L-D$ 
satisfies $S_k(L) \geq S_k(D)$.
\end{conj}

\paragraph{}
Here are some attempts:
Fan's theorem tells that $S_k(A+B) \leq S_k(A)+S_k(B)$ for all $1 \leq k \leq n$ in general. 
See \cite{Weyl1949,Fan1949} and \cite{DuvalReiner} call this spectral majorization. 
Fan's theorem also implies $S_k(D) \leq S_k(L)+ S_k(D-L)$ but that gives only $S_k(L) \geq S_k(D)-S_k(D-L)$.  
A similar argument gives $S_k(L) \leq S_k(D) + S_k(L-D)$.
we measure that $S_k(L+D)$ and $S_k(L-D)$ are very close.

\paragraph{}
The matrix $L-D$ has positive diagonal entries and is non-negative but that does not imply that it
has no-negative eigenvalues. The matrix $\left[ \begin{array}{cc} 3  &4  \\ 4  & 3  \end{array} \right]$ 
for example has eigenvalues $7,-1$ even so it is positive and has all positive entries.
Schur's theorem tells that if $B_{ij} >=0, \forall i,j$, then $S_k(B) \geq 0$.
For a Dirac matrix, we also have $S_k(-D) \geq 0$ simply because the spectra of $D$ and $-D$ agree. 
The identity $k \leq S_k(L-D) \leq S_k(L) + S_k(D)$  shows that $S_k(L) \geq k-S_k(D)$.

\section{Lefschetz fixed point Theorem}

\paragraph{}
The {\bf discrete Lefschetz fixed point theorem} \cite{brouwergraph}
works for any {\bf simplicial dynamical systems} $(G,T)$, where
$T$ is a {\bf simplicial map} and $G$ is a simplicial complex. Unlike in the continuum, where 
one needs to add some assumptions like regularity or that $T$ only has finitely many fixed points,
there are no conditions whatsoever in the discrete. 
We review first the language of simplicial complexes as \cite{brouwergraph} formulated the theorem
for Whitney complexes of graphs, which is almost no loss of generality as every simplicial complex 
defines a graph. The proof in \cite{brouwergraph} which was closer to Hopf's proof in the continuum.
As mentioned in various expositions like \cite{DiscreteAtiyahSingerBott} 
before, it can be proven by {\bf heat deformation}. We give here a bit more details. 

\paragraph{}
A simplicial complex $G$ comes with the {\bf Alexandrov topology}, in which the stars 
$U(x) = \{ y \in G, x \subset y \}$ form a topological basis. 
A {\bf simplicial map} is a continuous map that originates from a 
map on vertices. The Lefschetz formula generalizes the Euler-Poincar\'e formula and implies 
the Brouwer's theorem assuring that a simplicial map on a
complex homotopic to $K_1$ has at least one fixed point. 
\footnote{A.E. Brouwer from the Brouwer conjecture is not related to
L.E.J. Brouwer from the Brouwer fixed point theorem. The later had no children of his own.}

\paragraph{}
If a map $T$ on a simplicial complex is continuous then it preserves the order structure. But a continuous
map is not necessarily a simplicial map in general: a simplicial map can not map smaller dimensional parts
to higher dimensional ones. For example, a constant map $G \to G, x \to z$ which maps to a constant 
positive dimensional $z \in G$ is continuous, but it is not a simplicial map. The Lefschetz theorem needs 
to access orientation of a simplex in an other simplex and so needs that $T$ comes from a simplicial map. 
It seems not impossible that there might be a generalization to continuous maps but we have not seen this
to work yet. 

\paragraph{}
Lets first define some parts which enter the Lefschetz fixed point theorem. 
If $T: G \to G$ is a simplicial map, a {\bf fixed point} $x \in G$ is a simplex that satisfies $T(x)=x$. 
Denote with $F=F_T$ the set of fixed points. Define $\omega(x)=(-1)^{{\rm dim}(x)}$.
If $T$ is a simplicial map, the {\bf index} of $x$
is defined as $i_T(x)=\omega(x) {\rm sign}(T|x)$. The map $T$ induces the {\bf Koopman
matrix} $\mathcal{U}_T f = f(T(x))$, which as a linear map on $l^2(G)$ and so a $n \times n$ matrix. 

\begin{lemma}
${\rm str}(\mathcal{U}_T)=\sum_{x \in F} i_T(x)$ 
\end{lemma}
\begin{proof} 
If a vertex $x$ is a fixed point of $T$
then $\mathcal{U}(e_x) = e_x$, meaning that $\mathcal{U}(x,x)=1$. 
If $T(x)=y$ and $x \neq y$, then $\mathcal{U}(x,y)= e_x e_y = 0$. 
The analysis now reduces to a situation where $G$ is a 
complete complex (the set of all non-empty subsets of a set with $n$ elements). 
Then, every permutation $T:G \to G$ is a simplicial map and the fixed points are
unions of cycles in the cycle decomposition of $T$. The index of a cycle $y$
is $1$, the index of a pair of cycles is $-1$, the index of a triplet of cycles is $1$ etc. 
The sum $\sum_{x \in F} i_T(x)$ is in this case the Euler characteristic of a complete complex
$K_l$, where $l$ is the number of cycles of the permutation defined by $T$, which agrees with
${\rm str}(\mathcal{U}_T)$. 
\end{proof}

\paragraph{}
In the special case, when $T$ is the identity, then $\mathcal{U}$ is the identity matrix and 
${\rm str}(\mathcal{U}_T) = \chi(G)$ and $i_T(x) = \omega(x)$. This is a situation where the 
agreement follows from the definitions. 

\paragraph{}
The {\bf Lefschetz number} $\chi_T(G)$ 
is defined to be the super trace of the map $T$ induced on the kernel of $H=D^2 = \oplus_{k=0}^q H_k$:
$$ \chi_T(G) = {\rm str}(\mathcal{U}_T | {\rm ker}(H)) = \sum_{j=0}^q (-1)^q {\rm tr}(\mathcal{U}_T | {\rm ker}(H_k))  \; . $$
It generalizes the {\bf Euler characteristic} $\chi(G)$, which is the situation when $T$ is the identity.
In that case, $\mathcal{U}_T u=u$ is the identity matrix and the super trace on the kernel of $H_k$ is the Betti number $b_k$.
The Lefschetz fixed point theorem then reduces to the Euler-Poincar\'e formula:

\begin{thm}[Lefschetz fixed point theorem]
$\sum_{x, T(x)=x} i_T(x) = \chi_T(G)$ for any
simplicial map $T$ on any simplicial complex $G$. 
\end{thm} 

\paragraph{}
To prove this, we first reduce the case of a simplicial map to an invertible simplicial map.
While not all continuous maps are simplicial maps, every
homeomorphism of a simplicial complex necessarily is a simplicial map. Proof:
if $v$ is a vertex then $T(v)$ must be a vertex too (as if it would be a positive dimensional 
simplex $x$, then $x=T(v)$ implied $T^{-1}(x)=v$). But then also $T^{-1}(w)=v$ for any vertex 
$w \subset x$, which would contradict invertibility.
A simplicial map is not necessarily open as it can map $G$ to a vertex $\{v\}$ and so map an 
open set (like $G$) to a closed set (like a singleton set $\{v\}$). Note that the following is 
not true for "simplicial complexes" as often considered in the literature as geometric realizations
of finite abstract simplicial complexes. (Even \cite{DehnHeegaard}, where finite abstract complexes
first appeared also look at geometric realizations \cite{DehnHeegaard}.)

\begin{lemma}
If $T$ is a simplicial map on a finite abstract simplicial complex $G$, then 
the attractor $K=\bigcap_{k \geq 1} T^k(G)$ of $T$ is a simplicial complex.
\end{lemma}
\begin{proof}
We only need to check that $T(G)$ of $G$ is closed so that by induction $T^k(G)$ is closed implying that the
intersection (it is a finite intersection) is closed. In order to see that $T(G)$ is closed, let $y = T(x)$ 
be in the image $T(G)$. By definition of a simplicial map, this is given as $y=\{ T(v), v \in x \}$. 
Let $z \subset y$ be an arbitrary subset, then look at $u=\{ v \in x, T(v) \in z \}$. 
This set is in $G$ because $G$ is a simplicial complex
and so closed under the operation of taking non-empty subsets. Consequently,
$z=T(u)$. We see that $z \in T(G)$. As every $z \subset y \in T(G)$ is also in $T(G)$, The set 
$T(G)$ is closed and so a simplicial complex. 
\end{proof} 

\paragraph{}
For an arbitrary map $T: G \to G$ on a finite set $G$, we have:

\begin{lemma}
The map restricted to the attractor $K = \bigcap_{k\geq 0} T^k(G)$ is a permutation. 
\end{lemma}
\begin{proof} Use induction for $n=|G|$.
It is clear for $n=1$. The induction step uses that if $T:G \to G$
is given, then $T(G)$ is $T$ invariant and either has the same number of elements or not. In the first case,
it is a permutation, otherwise we can use induction.
If $T$ is a simplicial map on a simplicial complex $G$, then the attractor $K$ of $T$ is closed.
The map $T$ is now invertible on $K$ and so a permutation. 
\end{proof} 

\paragraph{}
The proof of the Leschetz fixed point theorem can now can be reduced to homeomorphisms. This
is an argument which is close to what Hopf's proof did: 

\begin{lemma}[Reduction to Homeomorphism]
If $T$ is a simplicial map on $G$ and $K$ is its attractor of $T$, then 
$\sum_{x \in G, T(x)=x} i_T(x)= \sum_{x \in K, T(x)=x} i_T(x)$ and $\chi_T(G) = \chi_T(K)$. 
\end{lemma} 
\begin{proof}
If $T(x)=x$, then $x$ obviously is in the attractor because $T^k(x)=x$ for every $k \geq 1$. 
As for the second statement, let $x$ be a simplex that is in $G$ but not in $T(G)$. Let 
$e_1, \dots e_l$ be a basis in ${\rm ker}(H_k)$ and $\tilde{e_1} \dots \tilde{e_l}$ the 
vector for which $\tilde{e_j}(x)=0$ and $\tilde{e}_j(y)=e_j(y)$ otherwise. 
Now, ${\rm tr}(\mathcal{U}_T| {\rm ker}(H_k)) = 
\sum_j e_j \cdot \mathcal{U}_T(e_j)$ as $\mathcal{U}_T e_j$ has a zero entry at position $x$ we have
$\sum_j \tilde{e_j} \cdot \mathcal{U}_T(e_j)$. This is $\sum_j \mathcal{U}_T(\tilde{e_j}) \cdot e_j$. 
Because $K$ is $T$-invariant, also $\mathcal{U}_T(\tilde{e_j})$ has a zero entry at $x$.
The trace is now $\sum_j \mathcal{U}_T(\tilde{e}_j) \cdot \tilde{e}_j$, meaning that the simplex $x$
has become irrelevant for the Lefschetz number. 
Doing so for every element $x \in G \setminus K$ proves the second claim $\chi_T(G) = \chi_T(H)$.
\end{proof} 

\paragraph{}
The {\bf McKean-Singer symmetry} \cite{McKeanSinger} (see \cite{knillmckeansinger}
for the discrete case) assures that ${\rm str}(e^{-t H})=\chi(G)$ follows from ${\rm str}(H^n)=0$
for $n>0$:

\begin{lemma}[McKean Singer]
${\rm str}(H^n)=0$ for $n>0$ so that ${\rm str}(e^{-tH}) = \chi(G)$ is independent of $t$. 
\end{lemma} 

\begin{proof} 
The Dirac matrix $D$ produces a bijection between the
non-zero spectrum of the even forms and the non-zero spectrum of the odd forms: 
if $Lu=\lambda u$ and $u \in l^2(\bigcup_{k, {\rm even}} G_k)$ (even forms),
then $L (Du) = \lambda (Du)$ and $Du$ is in $l^2(\bigcup_{k, {\rm odd}} G_k)$ (odd forms).
(This argument does not work if $\lambda$ is a zero eigenvalue because then, 
$Du$ is the zero vector.) There is an isomorphism between the kernel on even
forms and odd forms only if the Euler characteristic is zero. 
\end{proof}

\begin{lemma}
Assume $T$ is homeomorphism of $G$ and $H=D^2$ is the Laplacian: \\
a) $\mathcal{U}$ commutes with $H$. \\
b) $\mathcal{U}$ commutes with the heat flow $e^{-t H}$. \\
c) $\mathcal{U}_t = e^{-t H} \mathcal{U}$ is block diagonal for all $t$
with the same blocks than $H$. 
\end{lemma} 
\begin{proof} 
If $T$ is a homeomorphism, then $T$ preserves the order structure and 
also the dimensions of the simplicies. It does not necessarily preserve the orders and
so the sign of a subsimplex in a given simplex. \\
a) The relation $\mathcal{U} d = \pm d \mathcal{U}$ holds by definition. It implies that
  $\mathcal{U} D = \pm D \mathcal{U}$ and so $\mathcal{U} H = H \mathcal{U}$. \\
b) We can make a Taylor expansion in $t$ of $e^{-t H}$ and use that $\mathcal{U}$ commutes 
with every power $H^n$. \\
c) Both $\mathcal{U}$ and $e^{-t H}$ preserve $k$-forms and so are block diagonal. 
\end{proof} 

\paragraph{}
The McKean symmetry now immediately works also in a dynamical setting:

\begin{lemma}
${\rm str}(H^n \mathcal{U})=0$ for $n>0$ so that ${\rm str}(e^{-tH} \mathcal{U})$
is also independent of $t$.
\end{lemma} 

\begin{proof} 
Diagonalize $H$ and look at an eigenvector $f$ on even forms with $H f = \lambda f \neq 0$
then $H (Df) = D Hf= D \lambda f = \lambda (Df)$ gives an
eigenvectors $Df$ on odd forms. Now look at $H^n U$. If $f$ is an
eigenvector on k-forms, then we especially know that $U f$ remains
a vector on k-forms. $\mathcal{U}$ commutes with $D$ on a non-zero eigenspace of $H \mathcal{U}$.
If $H^n \mathcal{U} f=\lambda f$. Then $H^n \mathcal{U} (Df) = H^n D \mathcal{U}f = D H^n \mathcal{U} f = 
D \lambda f = \lambda (Df)$. 
\end{proof} 

\paragraph{}
We can now complete the  proof of the Lefschetz fixed point theorem:

\begin{proof} 
The deformed Koopman operators $\mathcal{U}_t f = e^{- t H} \mathcal{U} f$ is for
$t=0$ equal to  $\mathcal{U}_0=\mathcal{U}$. For $t \to \infty$, it is the induced
map on harmonic forms. 
In the limit $t \to \infty$, where $P=\lim_{t \to \infty} e^{-t H}$ is the projection onto 
the kernel of $H$, it is the super trace of $\mathcal{U}$ on
the kernel of $H$. This was defined to be the Lefschetz number. 
\end{proof} 
 
\paragraph{}
{\bf Remark:} {\bf Finite topological spaces} form a slightly more general category than 
the set of {\bf finite pre-ordered sets} where {\bf monotone maps} are morphisms.
The later can be identified with the class of finite topological spaces that are $T_0$ spaces.
For example: for the {\bf Sierpinski space} $\mathcal{O}=\{\emptyset,\{1\},\{1,2\}\}$, the constant map
$T(x)=0$ is continuous but not open because $\{0\}$ is not open. 
Open sets are the upward-closed sets for the specialization pre-order 
$x \leq y$ if $x \in \overline{\{y\}}$. This means $x \leq y$ if and
only if $y \in U(x)$ if and only if $x \subset y$. 
On a simplicial complex, the specialization pre-order is the subset partial 
order $x \subset y$. A theorem of Barmak-McCord assures that
every finite $T_0$-spaces encodes the homotopy type of a finite abstract
simplicial complex. The Kolmogorov quotient gets from a finite topological space, a $T_0$ space
using the equivalence relation $x \sim y$ if $\overline{\{x\}}=\overline{\{y\}}$.

\paragraph{}
Examples:  \\
{\bf 1)} if  $G=K_n$ and $T$ be a rotation map, there is one fixed point $x=G$ and $i_T(x)=1$. 
{\bf 2)} Let $G=K_n$ and assume $T$ has $k$ cycles. Then there are $k$
fixed points. Let $y$ be a fixed simplex on which $T$ is a rotation
then $i_T(y)=\omega(y)$. Now $\sum_{y \in F} i_T(y) = \sum_{y \in F} (-1)^{\rm dim(y)}
= \sum_{y \in F} \omega(y)$. 
This is by definition the super trace of $U$. 

\paragraph{}
{\rm Remark:} The theorem generalizes to correspondences $T(x) = \bigcup T_i(x)$
which have the property that the inverse $S$ of $T$ is a simplicial map.
One just can look at the transformation $S$ then. 
The $\omega$ limit set of $S$ then
is the $\alpha$ limit set of $T$. This models $T(x)=-x^2$ on the
discrete circle $C_n$, which has trace $1$ on constant functions and 
trace $-1$ on 1-harmonic forms, so that the Lefschetz number is $2$. 
There are also always the two fixed points $0,-1$. 

\section{Dynamical connection and Dirac matrices}

\paragraph{}
Related to the Lefschetz fixed point theorem, one can look at
{\bf dynamical connection matrices} $L_T(x,y) = \chi( C(x) \cap C(T(y)) )$ 
and dynamical Green function matrices $g_T(x,y) = \omega(x) \omega(y) \chi( U(x) \cap U(T(y))$. 
These are no more symmetric matrices in general. 
A bit surprisingly, these still are unimodular matrices, even so the spectrum is now complex in general. 
We have the following generalization of the unimodularity theorem. We were pretty excited to 
see this to pan out in examples until we realized that the proof is easy:

\begin{thm}
$L_T(x,y)^{-1} = g_T(x,y)^*$. 
\end{thm}

\paragraph{}
This is easier to see if one makes the problem harder and assume
that $T$ is a general permutation of the basis. This can be implemented
as a permutation matrix $P$.  
The result follows from the very general formula in linear algebra: 

\begin{lemma}
If $A,B$ are $n \times n$ matrices and $P_T$ a permutation matrix 
defining a permutation of the basis.  Define $A_T=A P_T$ and $B_T=B P_T$. 
Then $A_T B_T^*$ is independent of $T$. 
\end{lemma}
\begin{proof} 
Just compute it $A_T B_T^* = A P^T (P^T)^* B = A B$. 
\end{proof} 

\paragraph{}
This obviously directly produces the theorem above after noticing
that $L_T = L P_T$ and $g_T = g P_T$. 
Also the {\bf potential energy theorem} generalizes

\begin{thm}
$\sum_{x,y} g_T(x,y) = \chi(G)$
\end{thm}

\paragraph{}
Also this follows directly from the usual energy theorem 
$\sum_{x,y} g(x,y) = \chi(G)$, noting that

\begin{lemma}
If $A$ is an $n \times n$ matrix and $P_T$ a permutation matrix  and
$A_T = A P_T$. Then $\sum_{x,y} A_T(x,y)$ is independent of $T$. 
\end{lemma} 
\begin{proof}
The transformation $T$ just reshuffles the sum. 
\end{proof} 

\paragraph{}
One can also wonder about {\bf dynamical Dirac matrices}. Define the 
{\bf dynamically deformed exterior derivative} $d_T(x,y) = s(x,T(y))$. We have

\begin{thm}
$H=(d_T+d_T^*)^2$ is independent of $T$. 
\end{thm}

\begin{proof}
Also this follows from the above simple lemma. The matrices $d_T d_T^*$
and $d_T^* d_T$ both do not depend on $T$ so that $D_T^2$ does not depend on $T$.
\end{proof}

\vfill \pagebreak

\section{Code}

\paragraph{}
Here is the code that generated Figure 1 and Figure 2:

\begin{tiny}
\lstset{language=Mathematica} \lstset{frameround=fttt}
\begin{lstlisting}[frame=single]
Generate[A_]:=If[A=={},A,Delete[Union[Sort[Flatten[Map[Subsets,A],1]]],1]];
Whitney[s_]:=Generate[FindClique[s,Infinity,All]];       L=Length;      T=Transpose;
Connection[G_]:=Table[If[L[Intersection[G[[i]],G[[j]]]]>0,1,0],{i,L[G]},{j,L[G]}];
F[G_]:=Delete[BinCounts[Map[L,G]],1];     S[x_]:=Signature[x];
s[x_,y_]:=If[SubsetQ[x,y]&&(L[x]==L[y]+1),S[Prepend[y,Complement[x,y][[1]]]]*S[x],0];
Dirac[G_]:=Module[{f=F[G],d,n=L[G]},d=Table[s[G[[i]],G[[j]]],{i,n},{j,n}]; d+T[d]];
K=RandomGraph[{6,12}];G=Whitney[K];n=Length[G];
A=Connection[G]; B=Dirac[G];J=Inverse[A];
EA = Reverse[Sort[Eigenvalues[1.0*A]]]; EJ= Reverse[Sort[Eigenvalues[1.0*J]]];
EB = Reverse[Sort[Eigenvalues[1.0*B]]];
ea = Table[Sum[EA[[j]],{j,k}],{k,n}]; ej = Table[Sum[EJ[[j]],{j,k}],{k,n}];
eb = Table[Sum[EB[[j]],{j,k}],{k,n}]; di = Table[k, {k, n}];
DA = Reverse[Sort[Table[L[Select[G, L[Intersection[G[[k]],#]]>0 &]],{k,n}]]];
DB = Reverse[Sort[Table[L[Select[G, SubsetQ[G[[k]],#] || SubsetQ[#,G[[k]]] &]],{k,n}]]];
ListPlot[{EA,EJ,EB},Joined->True,Filling->Axis]
ListPlot[{ea,ej,eb},Joined->True,Filling->Axis]
ListPlot[{EA,DA},Joined->True,Filling->Axis]
ListPlot[{EB,DB},Joined->True,Filling->Axis]
\end{lstlisting}
\end{tiny}

\paragraph{}
And here is some illustration of the Lefschetz fixed point theorem. 

\begin{tiny}
\lstset{language=Mathematica} \lstset{frameround=fttt}
\begin{lstlisting}[frame=single]
Generate[A_]:=If[A=={},A,Delete[Union[Sort[Flatten[Map[Subsets,A],1]]],1]];
Whitney[s_]:=Generate[FindClique[s,Infinity,All]];  L=Length;  s[x_]:=Signature[x];
s[x_,y_]:=If[SubsetQ[x,y]&&(L[x]==L[y]+1),s[Prepend[y,Complement[x,y][[1]]]]*s[x],0]; 
W[k_]:=Map[Sort,Select[G,L[#]==k &]]; 
Dirac[G_]:=Table[s[G[[i]],G[[j]]]+s[G[[j]],G[[i]]],{i,n},{j,n}]; 
Str[X_]:=Sum[X[[k,k]]*(-1)^(L[G[[k]]]-1),{k,n}];
g=CompleteGraph[{3,4,2}]; GraphJoin[CycleGraph[4], CycleGraph[5]];
G=Whitney[g]; n=L[G]; V=Flatten[Select[G,L[#]==1 &]]; DD=Dirac[G]; HH=DD.DD;
AutQ[T_]:=Module[{G1=W[2]},G1==Sort[Map[Sort,(G1 /. Table[V[[k]]->T[[k]],{k,L[V]}])]]];
S[T_]:=Table[V[[j]]->T[[j]],{j,L[V]}]; Q=Transpose[NullSpace[HH]]; 
P=Q.Inverse[Transpose[Q].Q].Transpose[Q]; Aut=Select[Permutations[V],AutQ[#] &];
U[T_]:=Table[x=G[[k]];y=G[[l]];z=y /.Table[V[[j]]->T[[j]],{j,L[V]}];
          If[x==Sort[z],s[z],0],{k,n},{l,n}];
f[T_]:=Module[{J=S[T]},Fix=Select[G,Sort[# /.J]==#&];ind[x_]:=(-1)^(L[x]-1)*s[x/.J];
  {If[Fix=={},0,Total[Map[ind,Fix]]],Str[P.U[T]]}];
T=RandomChoice[Aut]; J=S[T]; UU=U[T]; Fix=Select[G,Sort[# /. J]==# &]; 
{Total[Map[ind,Fix]], Str[UU],Str[P.UU]}
R=Map[f,Aut]; {Union[R],Total[R]/L[R]}
\end{lstlisting}
\end{tiny}

\paragraph{}
Finally we look at some code which illustrates the episode on dynamical matrices $D_t,L_T,g_T$. 

\begin{tiny}
\lstset{language=Mathematica} \lstset{frameround=fttt}
\begin{lstlisting}[frame=single]
Generate[A_]:=If[A=={},A,Delete[Union[Sort[Flatten[Map[Subsets,A],1]]],1]]; 
Whitney[s_]:=Generate[FindClique[s,Infinity,All]]; L=Length; s[x_]:=Signature[x]; 
s[x_,y_]:=If[SubsetQ[x,y]&&(L[x]==L[y]+1),s[Prepend[y,Complement[x,y][[1]]]]*s[x],0];
F[G_,k_]:=Map[Sort,Select[G,L[#]==k&]]; Euler[G_]:=Total[Map[w,G]]; RP=RandomPermutation;  
S[G_,T_]:=Module[{V=Flatten[F[G,1]]},Table[V[[j]]->T[[j]],{j,L[V]}]];  IS=Intersection;
FindAut[G_]:=Module[{V=Flatten[F[G,1]],H=F[G,2],T},T=Permute[Range[L[V]],RP[L[V]]];
   While[Not[H==Sort[Map[Sort,(H/.S[G,T])]]],T=Permute[Range[L[V]],RP[L[V]]]]; T];
d[G_,W_]:=Table[x=G[[i]];y=G[[j]]/.Q;s[x,y],{i,L[G]},{j,L[G]}]; Clear[g,K]; 
Dirac[G_,Q_]:=d[G,G]+Transpose[d[G,W]]; W[G_,k_]:=Map[Sort,Select[G,L[#]==k&]];   
U[G_,x_]:=Select[G,SubsetQ[#,x] &];  w[x_]:=-(-1)^L[x];  w[x_,y_]:=w[x]*w[y]; 
g[G_,Q_]:=Table[x=G[[k]];y=G[[l]];z=y/.Q;w[x,y]*Euler[IS[U[G,x],U[G,z]]],{k,n},{l,n}];
K[G_,Q_]:=Table[If[L[IS[G[[i]],G[[j]]/.Q]]>0,s[G[[i]]]*s[G[[j]]],0],{i,L[G]},{j,L[G]}];

t=CompleteGraph[{2,3,4,1}]; G=Whitney[t]; n=L[G]; T=FindAut[G];  Q=S[G,T]; 
D1 = Dirac[G, Q]; D0 = Dirac[G,S[G,Range[Length[F[G,1]]]]];  gT=g[G,Q]; LT=K[G,Q]; 
D1.D1==D0.D0
Transpose[LT] == Inverse[gT]
Euler[G]==Total[Flatten[gT]]
\end{lstlisting}
\end{tiny}

\section{Remarks}

\paragraph{}
The Lefschetz fixed point theorem goes over to cohomologies of higher characteristics. Simplicial 
cohomology is just the first, linear cohomology. The next is quadratic cohomology, where one looks at the set
$G^{(2)}$ of all simplex pairs $(x,y)$ for which $x \cap y \neq \emptyset$. There is an exterior derivative
$d$ on functions on $G^{(2)}$ and so a Dirac matrix $D^{(2)}=d+d^*$, a Hodge Laplacian $H^{(2)}$ and 
cohomology in the form of the kernels of the blocks of this Laplacian, where the block $H^{(2)}_k$ 
belongs to functions on simplices for which ${\rm dim}( (x,y) ) = {\rm dim}(x) + {\rm dim}(y)=k$.
A simplicial map $T: G \to G$ induces a 
map $T^{(2)}: G^{(2)} \to G^{(2)}$ given by $T^{(2)}( (x,y) ) = (T(x),T(y))$.  A fixed point is a pair 
$(x,y)$ of intersecting simplices such that $T^{(2)}( (x,y) ) = (x,y)$, its index is 
$i_{T^{2}}(x,y) = i_T(x) i_T(y)$, the Lefschetz number is the super trace of $T^{(2)}$ on cohomology 
as before. The proof of the higher Lefschetz theorems are exactly the same as in the simplicial 
cohomology case: the Dirac matrix produces a McKean-Singer symmetry. One then looks at the
heat flow in the Koopman case $\mathcal{U}_{T^{(2)}} e^{- t H^{(2)}}$ which for $t=0$
gives the sum $\sum_{(x,y) \in F} i_{T^{(2)}}( (x,y) )$ of indices of fixed points and for 
$t \to \infty$ the Lefschetz number $\chi_{T^{(2)}}(G)$ for quadratic cohomology. 

\paragraph{}
The classical Birkhoff fixed point theorem from 1913 (revised a 100 years ago in 1926 \cite{BrNe75}) 
deals with area-preserving homeormorphisms of the cylinder $X=\mathbb{T} \times [0,1]$
that rotate the boundary in different directions. Since the classical Lefschetz number 
$\chi_T(X)$ in this case is zero, the Lefschetz theorem does not work and indeed, the area-preservation
property as well as the twist conditions have to enter somewhere. Disregarding any of them allows
to get examples without any fixed points, like rigid rotations or where one of the rotating invariant 
boundaries is an attractor. We contemplated in 2024 whether
quadratic cohomology could work to prove this last theorem of Poincar\'e: 
quadratic cohomology has been proven to be interesting already to distinguish the 
cylinder $X$ from the Moebius strip $M$ as $b^{(2)}(X) = (0, 0, 1, 1, 0)$ and 
$b^{(2)}(M) = (0,0,0,0,0)$ \cite{CohomologyWu}. 
This cohomology is more interesting already for a 2 dimensional complex, where
the Hodge matrix $H^{(2)}$ has 5 blocks because the maximal dimension of a simplex pair $(x,y)$ is $2+2=4$.
We argued then that the Lefschetz number for an area preserving map twisting is $2$. 
The area-preserving property should imply that the trace of the induced map on the $k=2$ cohomology is $1$.
The twisting condition then assure that the trace on the induced map on the $k=3$ cohomology is $-1$. The
later deals with maps on pairs $(x,y)$ where $x$ is a triangle and $y$ is an edge in $x$. 
If this argument holds, then the super trace is $2$ and imply that there are at least two fixed
points. One of the difficulties with such an approach is also that simplicial area preserving maps  on a 
cylinder are rather rigid. We have difficulty however to build examples of simplicial maps on a cylinder
that satisfy the area preservation and twist conditions. 

\paragraph{}
Every simplicial complex $G$, defines a {\bf Dirac zeta function} $\zeta(z)=\sum_j \lambda_j^{-s}(H)$,
where $\lambda_j$ are the non-zero eigenvalues of $H=D^2$. And then there is 
a  {\bf connection Zeta function} $\sum_j \lambda_j^{-s}(L^2)$, where $\lambda_j$ are the eigenvalues
of $L^2$, and where $L$ is the connection Laplacian.  One can also look at the {\bf Green zeta function}
$\sum_j \lambda_j^{-s}(g^2)$ For a one-dimensional complex $G$ and the Hodge zeta function, there 
is symmetry about the imaginary axes: the functional equation is $\zeta(z)=\zeta(-\overline{z})$. 
A comparision of the zeta functions in the case of matrices $D,L,g$ would be interesting. 

\paragraph{}
If $T$ is an automorphism of a complex $G$, the {\bf dynamical zeta function} is defined as
$$   \zeta_T(s) = \exp( \sum_{k=1}^{\infty} \chi_{T^k}(G) s^k/k) \; . $$
It is the generating function for all the {\bf Lefschetz numbers} $\chi_{T^k}$ of the
iterates of $T$. These are rational functions \cite{brouwergraph}. 

\paragraph{}
Given an isometry of a classical Riemannian manifold $M$, one can look at Vietris-Ribbs complexes
$G$ in $M$. It induces a continuous map on the complex if simplices are 
mapped into simplices. It is rarely a simplicial map however. We would like to have a frame work
in which not only the cohomology of the complex is the cohomology 
of the manifold but also the Lefschetz number of the map is the same than the
map in the finite. This needs to be explored more. 

\paragraph{}
The nonlinear discrete sine-Gordon partial differential equations
$L u = c \sin(u)$ and $H u = c \sin(u)$ have completely different 
behavior. In the first case, $u=0$ is the only solution for small $c$ while 
in the later case there are Harmonic forms $Hu = 0$ that solve the case $c=0$. 
We can say that in the connection case, we have a {\bf stability of the vacuum}. 
The fact that $u=0$ is stable follows from the fact that $L$ is invertible. It is an 
interesting phenomenon which persists when the simplicial complexes are infinite. 
The case of small $c$ in the case $Hu = c \sin(u)$ is subtle in the infinite case,
even for $\mathbb{Z}$. There is still some sort of stability but it is of KAM nature. 
For large $c$, one can use the standard implicit function theorem.
The {\bf anti-integrable limit} $c \to \infty$ gives many non-trivial solutions for
these two non-linear equations. \cite{AuAb90} We hope to investigate more what happens 
in the connection case. There is always a smallest $c$ after which 
$Lu = c \sin(u)$ starts to have solutions different from the vacuum. 

\section{Question}

\paragraph{}
The first question is illustrated in Figure~1: \\

{\bf A)} Is it true that $L \geq D$ or $L \geq g$ in the weak Loewner order?  \\

\paragraph{}
We also do not see yet an argument assuring that the inverse of $L$ has a maximal 
eigenvalue. It looks as if $g^n=L^{-n}$ is conjugated to a matrix for which all 
entries are positive.  \\

{\bf B)} Does $g=L^{-1}$ always have a unique maximal eigenvalue?  \\

\paragraph{}
While $D$ has spectrum that is symmetric about $0$, this is not the case for $L$.
The number of positive eigenvalues and the number of negative eigenvalues differ by the
Euler characteristic. The spectral radius of $g$ is interesting because $g(x,y)$ can 
be interpreted as a potential energy between $x$ and $y$. The spectral radius also 
controls the {\bf spectral gap} and so the stability of the vacuum $u=0$. \\

{\bf C)} Is the spectral radius of L always larger than the spectral radius of g?  \\

{\bf Update July 18, 2026:} Robert J. George (Caltech) pointed out that the 2-skeleton complex
of $K_5$ is an example for which the spectral radius of $g$ is larger. Experiments show that for
odd $k$, and large $n$, the k-skeleton complex of $K_n$ also has the property that
the maximal eigenvalue of $g$ is larger than the maximal eigenvalue of $L$. 
The skeleton complex $S_2(K_5)$ is homotopically a wedge of 4 spheres but not
homeomorphic to 4 spheres wedged at a point.       

\begin{tiny}
\lstset{language=Mathematica} \lstset{frameround=fttt}
\begin{lstlisting}[frame=single]
Generate[A_]:=If[A=={},A,Delete[Union[Sort[Flatten[Map[Subsets,A],1]]],1]];
Whitney[s_]:=Generate[FindClique[s,Infinity,All]];
Connection[G_]:=Table[If[Length[Intersection[G[[i]],G[[j]]]]>0,1,0],{i,Length[G]},{j,Length[G]}];
F[n_,k_]:=Module[{},G=Select[Whitney[CompleteGraph[n]],Length[#]<=k &];
  l=Connection[G]; g=Inverse[l]; Max[Eigenvalues[1.0*g]]-Max[Eigenvalues[1.0*l]]];
Do[Print[F[n,3]],{n,3,20}]
Do[Print[F[n,5]],{n,3,20}]
Do[Print[F[n,7]],{n,3,20}]
\end{lstlisting}
\end{tiny}

\paragraph{}
We need to extend the Lefschetz fixed point theorem from simplicial maps to situations,
where it can model homeomorphisms of manifolds with or without boundary. 
The following is a guiding challenge. \\

{\bf D)} Can the Birkhoff fixed point theorem be proven via quadratic cohomology? 

\bibliographystyle{plain}

\end{document}